\documentclass[12pt]{amsart}
\usepackage{caption}
\usepackage{subcaption}
\usepackage{graphicx}
\usepackage{ragged2e}
\usepackage{comment}
\usepackage{array} 
\usepackage{tikz}
\usepackage{latexsym,amsmath,amssymb,amsthm,amsfonts}
\usepackage[normalem]{ulem}

\usepackage[numeric]{amsrefs}

\usepackage{hyperref}

\usepackage[capitalise]{cleveref}

\usepackage{multido,pst-plot,pstricks,pst-node}
\newcommand{\R}{{\mathbb{R}}}

\renewcommand{\S}{{\mathbb{S}}}
\newcommand{\eps}{\varepsilon}

\newcommand{\E}{{\mathbb{R}}}

\newcommand{\B}{{\mathbb{B}}}

\newcommand{\Vol}{{\rm Vol}}

\renewcommand{\phi}{{{\varphi}}}

\newtheorem{theorem}{Theorem}

\newtheorem{corollary}[theorem]{Corollary}

\theoremstyle{remark}
\newtheorem{remark}{Remark}

\theoremstyle{claim}
\newtheorem{claim}[theorem]{Claim}

\title{Small volume bodies of constant width}

\author{A. Arman}
 \address{Department of Mathematics, University of Manitoba, Winnipeg, MB, R3T 2N2, Canada}
 \email{andrew0arman@gmail.com}
 \thanks{The first author was supported in part by a postdoctoral fellowship of the Pacific Institute for the Mathematical Sciences}

\author{A.\ Bondarenko}
 \address{Department of Mathematical Sciences, Norwegian University of Science and
 	Technology, NO-7491 Trondheim, Norway}
 \email{andriybond@gmail.com}
 \thanks{The second author was supported in part by Grant 334466 of the Research Council of Norway.}

\author{F.\ Nazarov}
 \address{Department of Mathematics, Kent State University, Kent OH 44242, USA}
 \email{nazarov@math.kent.edu}
 \thanks{The third author was supported by NSF Grant DMS-2154335.}

\author{A.\ Prymak}
 \address{Department of Mathematics, University of Manitoba, Winnipeg, MB, R3T 2N2, Canada}
 \email{prymak@gmail.com}
 \thanks{The fourth author was supported by NSERC of Canada Discovery Grant RGPIN-2020-05357.}

\author{D. Radchenko}
 \address{Laboratoire Paul Painlev\'{e}, Universit\'{e} de Lille, F-59655 Villeneuve d'Ascq, France}
 \email{danradchenko@gmail.com}
\thanks{The fifth author was supported by ERC Starting Grant No. 101078782.}


\begin{document}
\begin{abstract}
For every large enough $n$, we explicitly construct a body of constant width $2$ that has volume less than $0.9^n \text{Vol}(\mathbb{B}^{n}$), where $\B^{n}$ is the unit ball in $\mathbb{R}^{n}$. This answers a question of O.~Schramm. 
\end{abstract}
\maketitle

\section{Introduction}
\label{sec:intro}

A convex body $K$ in the $n$-dimensional Euclidean space $\R^n$ has constant width $w$ if the length of the orthogonal projection of $K$ onto any line is equal to $w$. 

By the isodiametric inequality in $\E^n$ (see, e.g., \cite{G}*{Th.~8.8, p. 152}), the greatest volume body of constant width 2 is the unit ball $\B^{n}$. On the contrast, the problem of finding the least volume body of constant width (Blaschke-Lebesgue problem) is more difficult and remains open for all $n\geq 3$, see for instance~\cite{CG}*{Ch. 7},~\cite{CFG}*{Part A 22},~\cite{MMO}*{p.334-335} for a history of the problem.

Let $\Vol(K)$ denote the volume of a body $K$. A body $K$ is said to have effective radius~$r$ if $\Vol(K)=\Vol(r\B^n)$. Let $r_n$ denote the smallest effective radius of a body of constant width~$2$ in $\E^n$. Evidently, $r_n\leq 1$. Schramm~\cite{Schr} established the first non-trivial lower bound $r_n\ge \sqrt{3+\tfrac{2}{n+1}}-1$ and asked (see also the survey of Kalai~\cite{K}*{Problem~3.4}) if there exists $\eps>0$ such that $r_n\leq 1-\eps$ for all $n\geq 2$.
We answer the question of Schramm in the affirmative by proving that $r_n<0.9$ for all sufficiently large $n$.

\begin{theorem}\label{thm:volumebound}
For every sufficiently large $n$, there is a body $M$ in $\R^n$ of constant width $2$ with $\Vol(M)< 0.9^n \Vol(\B^{n})$.
\end{theorem}

\section{Preliminaries and the construction of $M$}

For every positive integer $n$, let $\S^{n-1}=\{x\in \R^n \;:\; |x|=1\}$ be the unit sphere in $\R^{n}$. Let $\omega_n$ and $\Omega_n=\frac{\omega_{n}}{n}$ be, respectively, the surface area of $\S^{n-1}$ and the volume of $\B^{n}$. Define $\R^{n}_+$ to be the positive orthant, i.e.,
$$\R^{n}_+=\{(x_1,\ldots, x_n)\in \R^{n} \;:\;  x_i\geq  0 \; \text{for all} \; i\in\{1,\ldots, n\}\}.$$
Define $S=\S^{n-1}\cap \R^{n}_{+}$ and let 
\[
L:=\left(\sqrt{2}S\right)\cup\left((\sqrt{2}-2)S\right)\,.
\] 
Now we define
\begin{equation}\label{eq:M}M:=\bigcap_{x\in L}(x+2\B^n).\end{equation}

Figure~\ref{fig:pictures} is an illustration of a body $M$ in dimensions $n=2, 3$. The colors represent different quadrants/orthants. \footnote{We refer a reader to \href{http://prymak.net/const-width-3d.html}{http://prymak.net/const-width-3d.html} for a 3d view of the body $M$ ($n=3$). }
\begin{figure}
    \centering
    \includegraphics[width=0.35\linewidth]{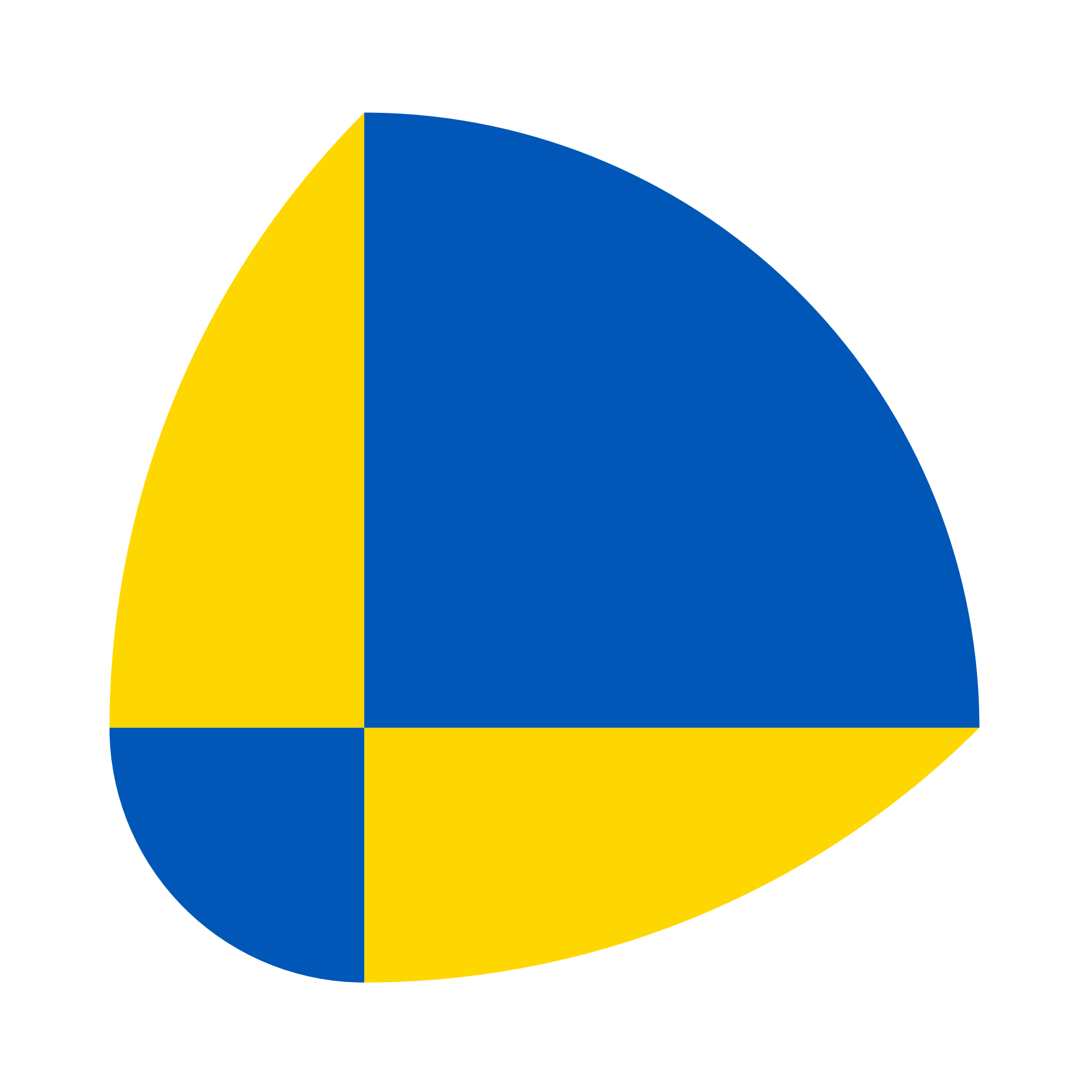}
    \qquad
    \includegraphics[width=0.35\linewidth]{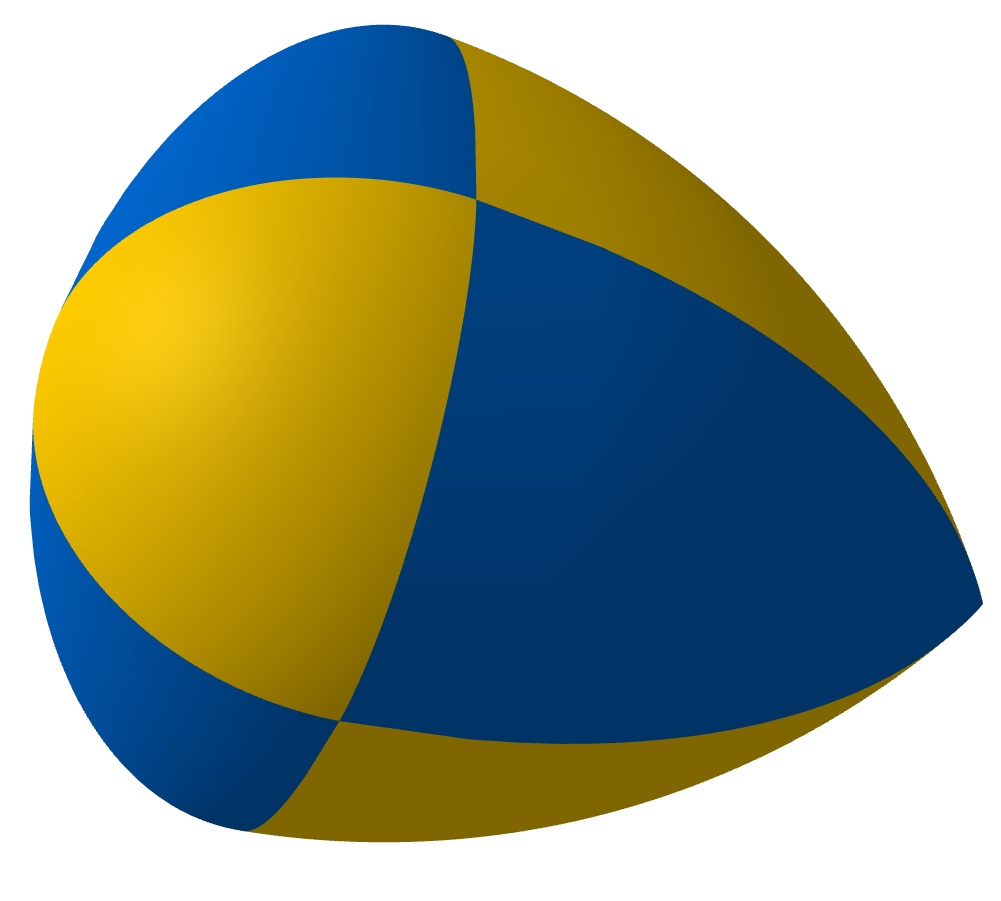}
    \caption{Illustration of the body $M$ in dimensions $n=2,3$.}
    \label{fig:pictures}
\end{figure}

Before we prove that $M$ is a body of constant width $2$ (Claim~\ref{claim:Mconstwidth}), we prove an auxiliary Claim~\ref{claim:MtoA}.
Note that any vector $v\in\R^n$ has a unique representation as $v=v_+-v_-$ with $v_{\pm}\in \R^{n}_+$  and $v_+\cdot v_- = 0$. More precisely, for a vector $v=(v_1,\ldots,v_n)$, the $i$-th coordinates of the vectors $v_{+}$ and $v_{-}$ are $\max\{v_i,0\}$ and $\max\{-v_i,0\}$ respectively. 

Define a set $$A:=\{(a,b)\in \R_+^2 \;\colon\; a^2+(b+\sqrt{2})^2\le 2^2\}.$$ 
\begin{claim}\label{claim:MtoA} If $v\in M$, then $(|v_+|,|v_-|)\in A$.
\end{claim}
\begin{proof} 
Define for every nonzero $x\in \R^{n}_+$ the vector $\widehat{x}=\frac{1}{|x|}x\in S$. Claim~\ref{claim:MtoA} holds trivially for $v=0$, so assume that $v\in \R^{n}$ is a nonzero vector.

If $v\in \R^{n}_{+}$, then $v=v_{+}$ and the distance from $v$ to $(\sqrt{2}-2)\widehat{v}\in L$ is $|v|+2-\sqrt{2}$, and so $|v|\leq \sqrt{2}$ and $(|v_{+}|, |v_{-}|)=(|v|, 0)\in A$. Otherwise, $\widehat{v}_{-}$ is well-defined, and the squared distance from $v$ to $\sqrt{2}\widehat{v}_{-}\in L$ is $|v_{+}|^2+(|v_{-}|+\sqrt{2})^2$, so $(|v_{+}|, |v_{-}|)\in A$. \end{proof}

\begin{claim}\label{claim:Mconstwidth} The set $M$, as defined in (\ref{eq:M}), is a body of constant width $2$. Moreover $M=\widetilde{M}$, where $$\widetilde{M}=\{v-w \;:\; v,w\in\R^{n}_{+}, \; (|v|,|w|)\in A\}.$$
\end{claim}
\begin{proof}
First we show that $\widetilde{M}$ has diameter at most $2$. Indeed, let $v_1-w_1, v_2-w_2 \in \widetilde{M}$. Then 
\begin{align*}
|(v_1-w_1)-(v_2-w_2)|^2&=|(v_1+w_2)-(v_2+w_1)|^2\leq |v_1+w_2|^2 +|v_2+w_1|^2\\
& \leq (|v_1|+|w_2|)^2+(|v_2|+|w_1|)^2=:d^2.
\end{align*}

Now, let $A'$ denote the set obtained from $A$ by reflecting it about the line $y=-x$, see Figure~\ref{fig:1}. Then $d$ is the distance between $(|v_1|,|w_1|)\in A$ and $(-|w_2|,-|v_2|)\in A'$. Both $A$ and $A'$ are contained in a Reuleaux triangle $\mathcal{R}$ with vertices $(\sqrt{2},0)$, $(0,-\sqrt{2})$, as depicted in Figure~\ref{fig:1}. So $(|v_1|,|w_1|), (-|w_2|,-|v_2|) \in \mathcal{R}$, and, since $\mathcal{R}$ has diameter $2$, we conclude that $d^2\leq 4$. Hence the diameter of $\widetilde{M}$ is at most $2$. 

By Claim~\ref{claim:MtoA}, $M\subseteq \widetilde{M}$. On the other hand, since $L\subseteq\widetilde{M}$ and $\widetilde{M}$ has diameter at most 2, we have
$$\widetilde{M}\subseteq \bigcap_{x\in \widetilde{M}}(x+2\B^{n})\subseteq \bigcap_{x\in L}(x+2\B^{n})=M.$$
Thus $\widetilde{M}=M$.

\begin{figure}[t]
\centering
			\begin{tikzpicture}[scale=2]
			\draw[->] (-1, 0) -- (2, 0) node[right] {$x$};
			\draw[->] (0, -2) -- (0,1) node[above] {$y$};
			\draw ([shift=(45:2)]0,-1.4142) arc (45:90:2);
                \draw[fill] (0,-1.4142) circle(0.7pt) node[left] {$(0,-\sqrt{2})$};
                \draw[fill] (1.4142,0) circle(0.7pt);
			\draw[fill] (1.4142+0.25,0)            node[above] {$(\sqrt{2},0)$};
			
                \draw ([shift=(180:2)]1.4142, 0) arc (180:225:2);
                \draw (0.7,0.2) node {$A$};
                \draw (-0.2,-0.7) node {$A'$};
                \draw[dotted] ([shift=(45:2)]0,-1.4142) arc (45:105:2);
                \draw[dotted] ([shift=(165:2)]1.4142,0) arc (165:225:2);
                \draw[dotted] ([shift=(-75:2)]-0.5176,0.5176) arc (-75:-15:2);
                \draw (0.7,-0.7) node {$\mathcal{R}$};
			\end{tikzpicture}

  \caption{Sets $A$, $A'$, and a Reuleaux triange $\mathcal{R}$ that contains both $A$ and $A'$.}
  \label{fig:1}
\end{figure}

Now we show that the width of $\widetilde{M}$ in any direction $\theta\in \S^{n-1}$ is at least $2$. Write $\theta=\theta_{+}-\theta_{-}$. Then $|\theta_{+}|^2+|\theta_{-}|^2=1$. Since the width in the direction $\theta$ is the same as in the direction $-\theta$, we can assume without loss of generality that $|\theta_{-}|\geq |\theta_{+}|$, and so $|\theta_{-}|\geq \frac{1}{\sqrt{2}}$. Now, both vectors $2\theta_{+}-(2|\theta_{-}|-\sqrt{2})\widehat{\theta}_{-}$ and $\sqrt{2}\widehat{\theta}_{-}$ are in $\widetilde{M}$, and the difference between the two vectors is $2(\theta_{+}-\theta_{-})=2\theta$. Hence the width of $\widetilde{M}$ in direction $\theta$ is at least 2.

So $\widetilde{M}=M$ is a convex body of constant width 2. \end{proof}

\section{Estimate for $\Vol(M)$}

In order to prove the Theorem, we will prove that there is a positive $\sigma<0.9$ such that
\begin{equation}\label{eq:main}
\Vol(M)\leq (n+1)\sigma^n \Omega_n.
\end{equation}
We will estimate the volume of $M$ in each of the $2^n$ coordinate orthants. We say that an orthant $Q$ is a $(k, n-k)$ orthant if $Q$ consists of the points with exactly $k$ positive and $n-k$ negative coordinates. 

If $Q$ is the $(n,0)$ orthant, we have $\Vol(M\cap Q)=\frac{1}{2^{n}}\left(\sqrt{2}\right)^n\Omega_{n}$, as $M\cap Q=\sqrt{2}\B^{n}\cap\R^{n}_+$. Similarly, if $Q$ is the $(0,n)$ orthant, we have $\Vol(M\cap Q)=\frac{1}{2^{n}}\left(2-\sqrt{2}\right)^n\Omega_{n}$.

Let now $Q$ be a $(k,n)$ orthant with $1\leq k\leq n-1$. Then, by Claim~\ref{claim:MtoA},
$$\Vol(M\cap Q)\leq \Vol(\{v\in Q \; : \; (|v_+|, |v_-|)\in A\}).$$
Let $Q_{+}$ and $Q_{-}$ be the positive orthants in $\R^{k}$ and $\R^{n-k}$ corresponding to the positive and the negative coordinates of $Q$, respectively. If we consider an infinitesimal rectangle $R=[a,a+da]\times[b,b+db]$ with $R\subset A$, the set $\{v \in Q \;:\; (|v_+|, |v_-|)\in R \}$ corresponds to the Cartesian product of the portions of the spherical shells $\{v'\in Q_+ \;:\;|v'|\in[a,a+da]\}\times\{v''\in Q_- \;:\;|v''|\in[b,b+db]\}$, and the corresponding shell volumes in these orthants are $\frac{\omega_k}{2^{k}}a^{k-1}\;da$ and $\frac{\omega_{n-k}}{2^{n-k}}b^{n-k-1}\;db$. Therefore, 
\begin{align*}\Vol(M\cap Q)&\leq \iint\limits_{A} \frac{\omega_k}{2^{k}}\frac{\omega_{n-k}}{2^{n-k}}a^{k-1}b^{n-k-1}\;da\; db=\frac{1}{2^n}k(n-k) \; \Omega_{k}\Omega_{n-k} \iint\limits_{A} a^{k-1}b^{n-k-1}\;da\; db.\end{align*}

Since the number of $(k,n-k)$ orthants is exactly $\binom{n}{k}$, we obtain the estimate
\begin{equation}\label{eq:1}
    \Vol(M)\leq \frac{\Omega_{n}}{2^{n}}\left((\sqrt{2})^n+(2-\sqrt{2})^n+ \sum_{k=1}^{n-1} k (n-k) \binom{n}{k} \frac{\Omega_{k}\Omega_{n-k}}{\Omega_{n}} \iint \limits_{A} a^{k-1}b^{n-k-1}\;da\; db\right).
\end{equation}

Now, let $\alpha, \beta>0$ be some real numbers such that the triangle $T_{\alpha,\beta}:=\{a,b\geq 0 \; : \; \frac{a}{\alpha}+\frac{b}{\beta}\leq1 \}$ contains the set $A$ (see Figure~\ref{fig:triangle}). Note that the condition $A\subseteq T_{\alpha, \beta}$ is equivalent to the statement that the distance from the line $\frac{a}{\alpha}+\frac{b}{\beta}=1$ to $(0,-\sqrt{2})$ is at least $2$, i.e., $\alpha(\beta+\sqrt{2})\geq 2\sqrt{\alpha^2+\beta^{2}}$. 

\begin{figure}[h]
\centering
			\begin{tikzpicture}[scale=2.5]
			\draw[->] (-0.5, 0) -- (2, 0) node[right] {$x$};
			\draw[->] (0, -0.5) -- (0,1.5) node[above] {$y$};
			\draw ([shift=(45:2)]0,-1.4142) arc (45:90:2);
			\draw (1.5,0)--(0,0.99);
                \draw[fill] (1.5,0) circle (0.5pt) node[below] {$(\alpha,0)$};
                \draw[fill] (0,0.99) circle (0.5pt) node[left] {$(0, \beta)$};
                \draw (0.75,0.5) node[above] {$T_{\alpha,\beta}$};
                \draw (0.7,0.2) node {$A$};
			\end{tikzpicture}

  \caption{Triangle $T_{\alpha,\beta}$ that contains $A$.}
  \label{fig:triangle}
\end{figure}

Then for all $k\in \{1, \ldots, n-1\}$, we have 
$$\iint \limits_{A} a^{k-1}b^{n-k-1}\;da\; db\leq \iint \limits_{T_{\alpha,\beta}} a^{k-1}b^{n-k-1}\;da\; db= \alpha^k\beta^{n-k}\iint \limits_{T_{1,1}} a^{k-1}b^{n-k-1}\;da\; db.$$
Now, note that $k(n-k)\iint \limits_{T_{1,1}} a^{k-1}b^{n-k-1}\;da\; db =\binom{n}{k}^{-1}$, see Claim~\ref{claim:Ikl} in the Appendix for the proof, and so 
$$k (n-k) \binom{n}{k} \frac{\Omega_{k}\Omega_{n-k}}{\Omega_{n}} \iint \limits_{A} a^{k-1}b^{n-k-1}\;da\; db \leq \frac{\alpha^{k}\Omega_{k} \cdot \beta^{n-k}\Omega_{n-k}}{\Omega_{n}}.$$

Since $\alpha\geq \sqrt{2}$ and $\beta \geq 2-\sqrt{2}$,~(\ref{eq:1}) implies   
\begin{equation}\label{eq:2}
    \Vol(M)\leq \frac{\Omega_{n}}{2^{n}}\left(\sum_{k=0}^{n} \frac{\alpha^{k}\Omega_{k} \cdot \beta^{n-k}\Omega_{n-k}}{\Omega_{n}}\right).
\end{equation}
Finally, $\alpha \B^{k} \times \beta \B^{n-k}\subseteq \sqrt{\alpha^{2}+\beta^{2}}\B^{n}$, which implies $\frac{\alpha^{k}\Omega_{k} \cdot \beta^{n-k}\Omega_{n-k}}{\Omega_{n}}\leq \left(\sqrt{\alpha^2+\beta^2}\right)^n$, and so for every $(\alpha, \beta)$ with $A\subseteq T_{\alpha,\beta}$, we have 
\begin{equation}\label{eq:3}
    r_{n}\leq \frac{1}{2}(n+1)^{1/n}\sqrt{\alpha^2+\beta^2}.
\end{equation}

Let $$s=\min\{ \sqrt{\alpha^2+\beta^2}\;:\; \alpha,\beta > 0, \; \alpha(\beta+\sqrt{2})\geq 2\sqrt{\alpha^2+\beta^2}\}.$$ Then 
\begin{equation}\label{eq:4}
    r_n\leq \frac{1}{2}(n+1)^{1/n}s.
\end{equation}

A human verifiable proof of $s< 1.8$ is obtained by choosing $\alpha=1.5$ and $\beta=0.7 \sqrt{2}$. For such values we have $\alpha^{2}+\beta^{2}=3.23<3.24=1.8^2$. On other hand, $\alpha(\beta+\sqrt{2})>1.5\cdot 1.7 \cdot \sqrt{2}>3.6\geq 2\sqrt{\alpha^2+\beta^2}$. The middle inequality is equivalent to $1.7\sqrt{2}>2.4$, i.e., $2.89\cdot 2>5.76$.

So $s< 1.8$ and inequality~(\ref{eq:4}) finishes the proof of the Theorem.

\begin{remark}
We will verify in the Appendix that $\frac12 s$ is the least positive root of the equation $8x^6 - 76x^4 + 54x^2 + 1=0$, whose numerical value is $0.89071\ldots$, see Claim~\ref{claim:polynom}.
\end{remark}

\begin{corollary}
There exists $\varepsilon>0$ such that $r_n\leq 1-\varepsilon$ for all $n\geq 2$.
\end{corollary}
\begin{proof}
By Theorem~\ref{thm:volumebound}, there is $n_0$ such that for $n>n_0$ we have $r_n< 0.9$. For every $n\in \{2,\ldots, n_0\}$, the body $M_n$ that we constructed is a body of constant width $2$ in $\R^{n}$ different from $\B^{n}$. So by the equality part of the isodiametric inequality (see, e.g., \cite{G}*{Th.~8.8, p. 152}),  $\varepsilon_n:=1-\left(\frac{\Vol(M_n)}{\Vol(\B^{n})}\right)^{1/n}>0$. Therefore we can take $\varepsilon=\min\{0.1, \varepsilon_2, \ldots,\varepsilon_{n_0}\}>0$.
\end{proof}

\section*{Appendix}

\begin{claim}\label{claim:Ikl}
Let $T_{1,1}=\{(x,y)\in \R^{2}_+ \;:\; x+y\leq 1\}$. Then for all integers $n$ and $k\in\{1,\ldots, n-1\}$, we have $k(n-k)\iint \limits_{T_{1,1}} a^{k-1}b^{n-k-1}\;da\; db =\binom{n}{k}^{-1}$.
\end{claim}
\begin{proof}
First,
$$\iint \limits_{T_{1,1}} a^{k-1}b^{n-k-1}\;da\; db=\int_{0}^{1}\int_{0}^{1-a}a^{k-1}b^{n-k-1}\; db\;da=\frac{1}{n-k}\int_{0}^{1}a^{k-1}(1-a)^{n-k}\; da.$$
Now, for a variable $x$ consider  
\begin{align*}P(x)&=\sum_{k=1}^{n}\binom{n-1}{k-1}\left(\int_{0}^{1}a^{k-1}(1-a)^{n-k}\; da\right)x^{k-1}=\int_{0}^{1}(1-a+ax)^{n-1}\; da\\&=\frac{x^{n}-1}{(x-1)n}=\sum_{k=1}^{n}\frac{1}{n}x^{k-1}.
\end{align*}
Comparing the coefficients at $x^{k-1}$ in $P(x)$, we get $\int_{0}^{1}a^{k-1}(1-a)^{n-k}\; da =\frac{1}{n \binom{n-1}{k-1}}$, and so
\[k(n-k)\iint \limits_{T_{1,1}} a^{k-1}b^{n-k-1}\;da\; db=\frac{k}{n \binom{n-1}{k-1}}=\binom{n}{k}^{-1}. \qedhere\]
\end{proof}

\begin{claim}\label{claim:polynom} Let $s=\min\{ \sqrt{\alpha^2+\beta^2}\;:\; \alpha, \beta > 0, \; \alpha(\beta+\sqrt{2})\geq 2\sqrt{\alpha^2+\beta^2}\}$. Then $x=\frac{1}{2}s$ is the least positive root of $P(x):=8x^6 - 76x^4 + 54x^2 + 1=0$, which is also the only root of $P(x)$ on $(0,1)$. 
\end{claim}
\begin{proof}
It follows from the definition of $s$ that $x$ is the least positive number for which there exist $\alpha,\beta>0$ such that $\alpha^2+\beta^{2}=4x^2$ and $\alpha(\beta+\sqrt{2})\geq 4x$, or equivalently for which
\begin{equation}\label{eq:max}
\max \{\alpha(\beta+\sqrt{2}) \; :\; \alpha, \beta > 0, \; \sqrt{\alpha^2+\beta^2}=4x^2\}\geq 4x.\
\end{equation}

Now, the function $f(\beta)=\sqrt{4x^2-\beta^2}(\beta+\sqrt{2})$ satisfies $f^\prime(0)=2x>0$, $f(2x)=0$, and is positive on $[0,2x)$. So the maximum of $f$ is attained at a critical point on $(0,2x)$. The critical point satisfies the quadratic equation $2\beta^2+\sqrt{2}\beta-4x^2=0$, whose only positive root is  
$$b=-\frac{1}{2\sqrt{2}}+\sqrt{\frac{1}{8}+2x^2}.$$
The condition~(\ref{eq:max}) is now equivalent to the inequality $f(b)\geq 4x$, or $f^{2}(b)\geq 16x^2$, or  
$$(4x^2-b^2)(b^2+2\sqrt{2}b+2)\geq 16x^2.$$
Since $b^2=2x^2-\frac{1}{\sqrt{2}}b$, we can rewrite this as $$(2x^2+\frac{1}{\sqrt{2}}b)(2x^2+\frac{3}{\sqrt{2}}b+2)\geq 16x^2,$$
or, equivalently,
$$4x^4+4x^2+\left(4\sqrt{2}x^2+\sqrt{2}\right)b+\frac{3}{2}b^2\geq 16x^2.$$
Replacing $b^2$ with $2x^2-\frac{1}{\sqrt{2}}b$ once again, we get 
\begin{align*}
4x^4+\left(4\sqrt{2}x^2+\frac{1}{2\sqrt{2}}\right)b&\geq 9x^2,\\
 4x^4+\left(4\sqrt{2}x^2+\frac{1}{2\sqrt{2}}\right)\left(\frac{-1}{2\sqrt{2}}+\sqrt{\frac{1}{8}+2x^2}\right)&\geq 9x^2,\\
 4x^4-2x^2-\frac{1}{8}+\sqrt{\frac{1}{8}+2x^2}\left(4\sqrt{2}x^2+\frac{1}{2\sqrt{2}}\right)&\geq 9x^2.
\end{align*}
The last inequality can be rewritten as 
$$\sqrt{\frac{1}{8}+2x^2}\left(4\sqrt{2}x^2+\frac{1}{2\sqrt{2}}\right)\geq \frac{1}{8}+11x^2-4x^4.$$
Note that the right hand side is positive for $x\in(0,1)$ and the argument after~(\ref{eq:4}) implies $x<0.9$. So the last inequality can be squared, and we are looking for the least positive $x$ such that
$$\left(\frac{1}{8}+2x^2\right)\left(4\sqrt{2}x^2+\frac{1}{2\sqrt{2}}\right)^2\geq \left(\frac{1}{8}+11x^2-4x^4\right)^2, \text{i.e.},$$
$$\left(\frac{1}{8}+2x^2\right)\left(32x^4+4x^2+\frac{1}{8}\right)\geq 16x^8+121x^4+\frac{1}{64}-88x^6-x^4+\frac{11}{4}x^2, \quad \text{or}$$
$$64x^6+12x^4+\frac{3}{4}x^2+\frac{1}{64}\geq 16x^8-88x^6+120x^4+\frac{11}{4}x^2+\frac{1}{64},$$
which after canceling $\frac{1}{64}$, moving all terms to one side and dividing by $2x^2$ turns into $$P(x)=8x^6-76x^4+54x^2+1\leq 0.$$ 

Note that $P(0)=1>0$, $P(1)=-13<0$ and $P(+\infty)=+\infty$. So $P$ has at least one root on $(0,1)$ and at least one root on $(1,+\infty)$. Combining this with Descartes' rule of signs, we see that $P(x)$ has exactly two roots on $(0,\infty)$. Thus the least positive $x$ with $P(x)\leq 0$ is the least positive root of $P$, which is also the only root of $P$ on $(0,1)$. 
\end{proof}

\end{document}